\documentclass[12pt]{amsart}

  \baselineskip 1 mm

\usepackage{amsmath,amsthm,amsfonts,latexsym,amsopn,verbatim,amscd,amssymb}
\usepackage{hyperref}
\usepackage{graphics,color}
\usepackage{graphicx}
\usepackage{array,arydshln}
\usepackage{amsbsy}
\usepackage{pst-all}
\usepackage{pstricks}
\usepackage{graphics}
\usepackage{float}
\theoremstyle{plain}
\newtheorem{Thm}{Theorem}[section]
\newtheorem{Lem}[Thm]{Lemma}

\newtheorem{Cor}[Thm]{Corollary}
\theoremstyle{definition}

\newtheorem{Def}[Thm]{Definition}

\numberwithin{equation}{section}

\newcommand{\bnum}{\begin{enumerate}}
\newcommand{\enum}{\end{enumerate}}

\begin{document}
\title{Nilpotent Graph}
\author[D. K. Basnet, A. Sharma and R. Dutta]{Dhiren Kumar Basnet, Ajay Sharma and Rahul Dutta}
\address{\noindent D. K. Basnet, A. Sharma and R. Dutta \newline Department of Mathematical Sciences,\newline Tezpur University,\newline  Napaam-784028, Sonitpur,\newline Assam, India.}
\email{dbasnet@tezu.ernet.in, ajay123@tezu.ernet.in and dutta8rahul@gmail.com}


%


%
%
\begin{abstract}
 In this article, we introduce the concept of nilpotent graph of a finite commutative ring. The set of all non nilpotent elements of a ring is taken as the vertex set and two vertices are adjacent if and only if their sum is nilpotent. We discuss some graph theoretic properties of nilpotent graph.
\end{abstract}

\maketitle

%
%

\section{Introduction} \label{S:intro}

In this article, rings are finite commutative rings with non zero identity. The set of nilpotent elements of a ring $R$ and the $n\times n$ matrix ring over $R$ are denoted by $Nil(R)$  and $M_n(R)$ respectively. Here, by graph, we mean simple undirected graph. For a graph $G$, the set of vertices and the set of edges are denoted by $V(G)$ and $E(G)$ respectively. For a positive integer $n$, P. Grimaldi\cite{gfr} defined and studied various properties of the unit graph $G(\mathbb{Z}_n)$, of the ring of integers modulo $n$, with vertex set $\mathbb{Z}_n$ and two distinct vertices are adjacent if and only if their sum is a unit. Further in \cite{ug}, authors generalized $G(\mathbb{Z}_n)$ to unit graph $G(R)$, where $R$ is an arbitrary associative ring with non zero identity.

In this article we have introduced nilpotent graph $G(R)$ associated with a finite ring $R$. We define the nilpotent graph $G(R)$ of a ring $R$ taking $R\setminus Nil(R)$ as the vertex set and two vertices $x$ and $y$ are adjacent if and only if $x+y$ is a nilpotent element in $R$. The properties like girth, clique number, chromatic index, dominating number, spectrum, Laplacian spectrum and signless Laplacian spectrum of $G(R)$ have been studied.\par

 For this article, we mention some preliminaries about graph theory. Let $G$ be a graph. The degree of the vertex $v\in G$ is the number of edges adjacent with $v$, denoted by $deg(v)$. A graph $G$ is said to be connected if for any two distinct vertices of $G$, there is a path in $G$ connecting them. 
 Also girth of $G$, denoted by $gr(G)$ is the length of the shortest cycle in $G$ and if there is no cycle in $G$ then $gr(G)=\infty$. A complete graph is a simple undirected graph in which every pair of distinct vertices is connected by a unique edge. A bipartite graph $G$ is a graph whose vertices can be divided into two disjoint parts $V_1$ and $V_2$, such that $V(G) = V_1\cup V_2$ and every edge in $G$ has the form $e=(x,y) \in E(G)$, where $x\in V_1$ and $ y \in V_2$. A complete bipartite graph is a graph where every vertex of $V_1$ is connected to every vertex of $V_2$, denoted by $K_{m,n}$, where $|V_1|=m$ and $|V_2|=n$. A complete bipartite graph $K_{1,n}$ is called star graph. \par
A clique is a subset of vertices of a graph such that its induced subgraph is complete. A clique having $n$ number of vertices is called $n$-clique. The maximum clique of a graph is a clique such that there is no clique with more vertices. The clique number of a graph $G$ is denoted by $\omega (G)$ and defined by the number of vertices in the maximal clique of $G$.  An edge colouring of a graph $G$  is a map  $\xi : E(G) \rightarrow S$, where $S$ is a set of colours such that for all $e_1, e_2 \in E(G)$, if $e_1$ and $e_2$ are adjacent then $\xi(e_1) \neq \xi(e_2)$. The \textit{chromatic index} of a graph $G$ is denoted by $\chi^\prime(G)$ and is defined as the minimum number of colours needed for a proper colouring of $G$.
\section{Nilpotent graphs}
\begin{Def}
  The nilpotent graph of a ring $R$ denoted by $G(R)$ is defined by setting $R\setminus Nil(R)$ as the vertex set and two distinct vertices $x$ and $y$ are adjacent if $x+y$ is nilpotent. Here we are not considering any loop at a vertex in the graph.
\end{Def}
The nilpotent graphs of $\mathbb{Z}_{12}$ and $\mathbb{Z}_{18}$ are given below,

\begin{figure}[H] 
\begin{pspicture}(9,3)(-2,-.6)
\scalebox{.7}{
\rput(1,0){
\psdot[linewidth=.05](1,3)
\psdot[linewidth=.05](1,1)
\psdot[linewidth=.05](3,3)
\psdot[linewidth=.05](3,1)
\psdot[linewidth=.05](5,3)
\psdot[linewidth=.05](5,1)
\psdot[linewidth=.05](7,3)
\psdot[linewidth=.05](7,1)
\psdot[linewidth=.05](1,0)
\psdot[linewidth=.05](7,0)
\rput(1,3.3){$2$}
\rput(1,.7){$10$}
\rput(3,3.3){$4$}
\rput(3,.7){$8$}
\rput(5,3.3){$1$}
\rput(5,.7){$11$}
\rput(7,3.3){$5$}
\rput(7,.7){$7$}
\rput(.7,0){$3$}
\rput(7.3,0){$9$}

\psline(1,3)(1,1)(3,1)(3,3)(1,3)
\psline(5,3)(5,1)(7,1)(7,3)(5,3)
\psline(1,0)(7,0)
}}
\end{pspicture}
\caption{Nilpotent graph of $\mathbb{Z}_{12}$.}
\end{figure}
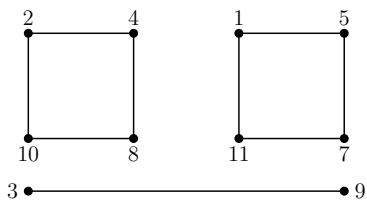

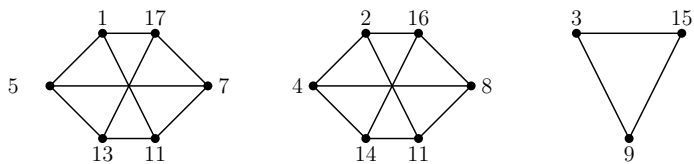
\begin{figure}[H] 
\begin{pspicture}(12,3)(-2,-.6)
\scalebox{.7}{
\rput(1,0){
\psdot[linewidth=.05](1,3)
\psdot[linewidth=.05](2,3)
\psdot[linewidth=.05](0,2)
\psdot[linewidth=.05](3,2)
\psdot[linewidth=.05](1,1)
\psdot[linewidth=.05](2,1)
\psdot[linewidth=.05](6,3)
\psdot[linewidth=.05](7,3)
\psdot[linewidth=.05](5,2)
\psdot[linewidth=.05](8,2)
\psdot[linewidth=.05](6,1)
\psdot[linewidth=.05](7,1)
\psdot[linewidth=.05](10,3)
\psdot[linewidth=.05](12,3)
\psdot[linewidth=.05](11,1)
\rput(1,3.3){$1$}
\rput(2,3.3){$17$}
\rput(-.7,2){$5$}
\rput(3.3,2){$7$}
\rput(1,.7){$13$}
\rput(2,.7){$11$}
\rput(6,3.3){$2$}
\rput(7,3.3){$16$}
\rput(4.7,2){$4$}
\rput(8.3,2){$8$}
\rput(6,.7){$14$}
\rput(7,.7){$11$}
\rput(10,3.3){$3$}
\rput(12,3.3){$15$}
\rput(11,.7){$9$}
\psline(1,3)(2,3)(3,2)(2,1)(1,1)(0,2)(1,3)
\psline(6,3)(7,3)(8,2)(7,1)(6,1)(5,2)(6,3)
\psline(10,3)(12,3)(11,1)(10,3)
\psline(1,3)(2,1)
\psline(2,3)(1,1)
\psline(0,2)(3,2)
\psline(6,3)(7,1)
\psline(7,3)(6,1)
\psline(5,2)(8,2)
}}
\end{pspicture}
\caption{Nilpotent graph of $\mathbb{Z}_{18}$.}
\end{figure}
From the above two graphs, it is clear that $gr(G(\mathbb{Z}_{12}))$ and $gr(G(\mathbb{Z}_{18}))$ are $2$ and $3$ respectively.
\begin{Lem}\label{L1}
  Let $G(R)$ be the nilpotent graph of a ring $R$. Then for $x\in V(G(R))$ we have the following:
  \begin{enumerate}
    \item If $2x\in N(R)$ then $deg(x)=|N(R)|-1$.
    \item If $2x\notin N(R)$ then $deg(x)=|N(R)|$.
  \end{enumerate}
\end{Lem}
\begin{proof}
  Proof is similar to Lemma $2.4$ \cite{ncg}.
\end{proof}
\begin{Cor}
  Every graph $G(\mathbb{Z}_n)$ with $n=2^k$ where $k\geq 1$ is a complete graph.
\end{Cor}
\begin{Cor}
  Every graph $G(\mathbb{Z}_n)$ such that $Nil(\mathbb{Z}_n)=\{0\}$ has chromatic number equal to $2$.
\end{Cor}
\begin{proof}
  From Lemma \ref{L1}, we have $deg(x)=1$, for all $x\in V(G(\mathbb{Z}_n))$. Hence chromatic number is 2.
\end{proof}
\begin{Thm}\label{main}
  Every graph $G(\mathbb{Z}_{\overline{n}})$, where $\overline{n}=p^nq^m$, $n>m$ and $p,q$ are distinct primes has the following properties:
  \begin{enumerate}
    \item Let $p=2$ and $q$ be an odd prime then
    \begin{itemize}
      \item If $|Nil(\mathbb{Z}_{\overline{n}})|>2$ then $G(\mathbb{Z}_{\overline{n}})$ is not bipartite and has a complete subgraph of order $Nil(\mathbb{Z}_{\overline{n}})$ and its complement is bipartite.
      \item If $|Nil(\mathbb{Z}_{\overline{n}})|\leq 2$ then $G(\mathbb{Z}_{\overline{n}})$ is bipartite.
    \end{itemize}

    \item If $p,q>2$ then $G(\mathbb{Z}_{\overline{n}})$ is bipartite.
  \end{enumerate}
\end{Thm}
\begin{proof}
  $(i)$ Let $p=2$ and consider $|Nil(\mathbb{Z}_{\overline{n}})|=k$. Now we partition our vertex set given by $P_1=\{1,2,\cdot\cdot\cdot, pq-1\}$, $P_2=\{pq+1,pq+2,\cdot\cdot\cdot,2pq-1\}$, $\cdot\cdot\cdot$, $P_{k_l}=\{pq(p^{n-1}q^{m-1}-1)+1, pq(p^{n-1}q^{m-1}-1)+2,\cdot\cdot\cdot, p^nq^m-1\}$. Observe that $k_l=k=|Nil(\mathbb{Z}_{\overline{n}})|$. As $p=2$, we see that $|P_1|=|P_2|=\cdot\cdot\cdot=|P_k|=pq-1=$odd. Since $|P_1|$ is odd so, there exists $r_1\in P_1$ such that $2r_1=pq$. Similarly $r_2\in P_2$ such that $2r_2=3pq$ and we get $r_1+r_2=2pq\in Nil(\mathbb{Z}_{\overline{n}})$. Similarly as above we can easily show that for any $P_i$, there exists $r_i\in P_i$ such that $2r_i=(2i-1)pq$. Let $C=\{r_1,r_2,\cdot\cdot\cdot,r_k\}$. Then for any pair $r_i,r_j\in C$, we get $r_i+r_j=(i+j-1)pq\in Nil(\mathbb{Z}_{\overline{n}})$. Hence induced subgraph of $C$ forms a complete subgraph of $G(\mathbb{Z}_{\overline{n}})$ and $G(\mathbb{Z}_{\overline{n}})$ is not bipartite. Now by Lemma \ref{L1}, the complete subgraph induced by $C$ is not connected with any vertex in $G\setminus C$. Consider $B_i=\{rpq+i\,:\, 0\leq r \leq (p^{n-1}q^{m-1}-1)\}$ and $B_{-i}=\{spq-i\,:\, 0\leq s \leq (p^{n-1}q^{m-1}-1)\}$, where $1\leq i< \frac{pq}{2}$. It is clear that every element of $B_i$ is adjacent to every element of $B_{-i}$. For $x\in B_i$ if $x+y\in Nil(R)$ then $y=n_1-x=(n_1-n_2)-i\in B_{-i}$, where $x=n_2+i$ and $n_1,n_2\in Nil(R)$. Let $A=B_1\cup B_2\cup \cdot\cdot\cdot \cup B_{\frac{pq}{2}-1}$ and $B=B_{-1}\cup B_{-2}\cup \cdot\cdot\cdot \cup B_{-(\frac{pq}{2}-1)}$. Observe that no two elements of $A$ are adjacent as, $i+j\notin Nil(R)$, for $1\leq i,j<\frac{pq}{2}$. Similarly no two elements of $B$ are adjacent and hence the result. If $|Nil(\mathbb{Z}_{\overline{n}})|\leq 2$ then clearly $G(R)$ is bipartite.
  \par
  $(ii)$ If $p$ and $q$ are two distinct odd primes, then $P_i$ is even for $1\leq i\leq k$. Now in $P_i$ we see that $1$ is adjacent to $ipq-1$, more generally $rpq+1$ is adjacent to $spq-1$, where $0\leq r \leq (p^{n-1}q^{m-1}-1)$ and $0\leq s \leq (p^{n-1}q^{m-1}-1)$. Let $B_1=\{rpq+1\,:\, 0\leq r \leq (p^{n-1}q^{m-1}-1)\}$ and $B_{-1}=\{spq-1\,:\, 0\leq s \leq (p^{n-1}q^{m-1}-1)\}$, then clearly each vertex of $B_1$ and each vertex of $B_{-1}$ are adjacent. But no two elements of $B_1$ and no two elements of $B_{-1}$ are adjacent. For $x\in B_1$ if $x+y\in Nil(R)$ then $y=n_1-x=(n_1-n_2)-i\in B_{-i}$, where $x=n_2+i$ and $n_1,n_2\in Nil(R)$, implies $G(B_1\cup B_{-1})$ is a subgraph which is bipartite. In general we can get $G(B_w\cup B_{-w})$ is a subgraph which is bipartite, where $1\leq w \leq [\frac{pq}{2}]$. Now we partition $G(\mathbb{Z}_{\overline{n}})$ into two parts viz. $A_1=B_1\cup B_2\cup\cdot\cdot\cdot\cup B_{[\frac{pq}{2}]}$ and $A_2=B_{-1}\cup B_{-2}\cup\cdot\cdot\cdot\cup B_{-[\frac{pq}{2}]}$. If possible let $a,b\in A$, such that $a+b\in Nil(\mathbb{Z}_{\overline{n}})$, then $a+b=(r+s)pq+(a_1+b_1)\in Nil(\mathbb{Z}_{\overline{n}})$, where $a_1,b_1\leq [\frac{pq}{2}]$, which is a contradiction to the fact that $a_1+b_1\in Nil(\mathbb{Z}_{\overline{n}})$, as $a_1,b_1< pq$, so no two vertices of $A_1$ are adjacent. Similarly we can show that no two vertices of $A_2$ are adjacent. Hence $G(\mathbb{Z}_{\overline{n}})$ is bipartite.
\end{proof}
\begin{Cor}\label{main2}
If $n=p_1^{r_1}p_2^{r_2}\cdot\cdot\cdot p_k^{r_k}$, where $p_i$'s are distinct odd primes, then $G(\mathbb{Z}_n)$ is bipartite and it consists of $[\frac{p_1p_2\cdot\cdot\cdot p_k}{2}]$ disjoint complete bipartite subgraphs $K_{|Nil(\mathbb{Z}_n)|, |Nil(\mathbb{Z}_n)|}$.
\end{Cor}
\begin{Cor}\label{main3}
  If $n=p_1^{r_1}p_2^{r_2}\cdot\cdot\cdot p_k^{r_k}$, where $p_1=2$ and $p_i$'s are distinct odd primes for $2\leq i \leq k$, then $G(\mathbb{Z}_n)$ consists of one complete subgraph of order $|Nil(R)|$ and its complement is bipartite, which is a union of disjoint complete bipartite subgraphs $K_{|Nil(\mathbb{Z}_n)|, |Nil(\mathbb{Z}_n)|}$.
\end{Cor}
\section{Girth}
\begin{Thm}
  For the graph of $G(\mathbb{Z}_n)$ the following hold:
  \begin{enumerate}
    \item If $n$ is odd then girth of $G(\mathbb{Z}_n)$ is $4$ provided $|Nil(\mathbb{Z}_n)|\geq 3$ and girth of $G(\mathbb{Z}_n)$ is infinite otherwise.
    \item Let $n$ be an even number. Then
     \begin{itemize}
       \item If $|Nil(\mathbb{Z}_n)|\geq 3$, then girth of $G(\mathbb{Z}_n)$ is $3$.
       \item If $|Nil(\mathbb{Z}_n)|=2$, then girth of $G(\mathbb{Z}_n)$ is $4$, provided $|R\setminus Nil(R)|>4$ otherwise girth of $G(\mathbb{Z}_n)$ is infinite.
       \item If $|Nil(\mathbb{Z}_n)|=1$, then girth of $G(\mathbb{Z}_n)$ is infinite.
     \end{itemize}
  \end{enumerate}
\end{Thm}
\begin{proof}
  \begin{enumerate}
    \item If $|Nil(\mathbb{Z}_n)|\geq 3$ then there exists $x\in \mathbb{Z}_n$ such that $n_1-x$, $n_1+x$, $n_2-x$ and $x$ are not nilpotent elements, where $0,n_1, n_2\in Nil(\mathbb{Z}_n)$ are distinct elements. So we have a $4$ cycle $x \,-\, (n_1-x)\,-\, (n_1+x)\,-\,(n_2-x)\,-\,x$, hence from Theorem \ref{main}, $gr(G(\mathbb{Z}_n))=4$. If $|Nil(\mathbb{Z}_n)|<3$, then $|Nil(\mathbb{Z}_n)|=1$ as $|Nil(\mathbb{Z}_n)|$ divides $n$.
    \item Proof is obvious from Theorem \ref{main}.
  \end{enumerate}
\end{proof}
\begin{Lem}\label{Gcor1}
Let $R$ be a commutative ring of odd order and $x\in R$. Then $x\in Nil(R)$ if $2x \in Nil(R)$.
\end{Lem}
\begin{proof}
Let $2x\in Nil(R)$. There exists $n\in \mathbb{N}$ such that $2^nx^n=0$. Since $(R,+)$ is an abelian group, as an abelian group, $Ord(x^n)$ divides $2^n$, so $Ord(x^n)=2^k$, for some $0\leq k \leq n$. But as an abelian group, $Ord(x^n)$ divides $|R|$, implies $k=0$. Hence $x^n=0$ $i.e.$ $x\in Nil(R)$.
\end{proof}
\begin{Lem}\label{Gmain1}
  Let $R$ be a commutative ring of even order. Then there exists $x\in R\setminus Nil(R)$ such that $2x\in Nil(R)$.
\end{Lem}
\begin{proof}
  Consider $|R|=2^km$, where $2$ does not divides $m$. Let $x=m.1$, where $m.1=(1+1+\cdot \cdot \cdot +1)(m-times)$. Observe that for $n\in \mathbb{N}$, $x^n=m^n.1$. If possible let $x=m.1=0$ then for any $a\in R$, $m.a=(1+1+\cdot \cdot \cdot +1).a=(m.1)a=0$, which is a contradiction as by cauchy theorem there exists an element in $R$ having order $2$ in $(R,+)$. So $x\neq 0$. Suppose $x^n=0$ for $n\in \mathbb{N}$ then $m^n.1=0$ implies $O(1)$ divides $m^n$, where $O(1)$ represents order of $1$ in $(R,+)$. Also $O(1)$ divides $2^km$, so $O(1)$ divides $m$ $i.e.,$ $m.1=x=0$ a contradiction. Hence $x\in R\setminus Nil(R)$. Now $(2x)^k=(2m)^k.1=(2^km^k).1=m^{k-1}((2^km).1))=0$ $i.e.,$ $2x\in Nil(R)$.
 \end{proof}
\begin{Thm}\label{GThmo}
  Let $R$ be a commutative ring of odd order and $|Nil(R)|\geq 3$ then $gr(G(R))$ is $4$ and $gr(G(R))$ is infinite otherwise.
\end{Thm}
\begin{proof}
  If possible let $gr(G(R))$ is $3$ then there exists a path $z_1\longrightarrow z_2 \longrightarrow z_3 \longrightarrow z_1$ in $G(R)$. $i.e.$ $\{2z_1, 2z_2, 2z_3\}\subset Nil(R)$, which is a contradiction by Lemma \ref{Gcor1}, as $z_i\notin Nil(R)$ for $1\leq i \leq 3$. \\
  Consider $R\setminus Nil(R)\neq \phi$  otherwise our graph is empty graph. Let $x\in R\setminus Nil(R)$. Since $|Nil(R)|\geq 3$, so there exists elements $0,n_1,n_2\in Nil(R)$ such that all are distinct. Clearly $x\longrightarrow n_1-x\longrightarrow n_2+x\longrightarrow n_2-x\longrightarrow x$ is a cycle in $G(R)$. Observe that all four elements $x, n_1-x, n_2+x, n_2-x$ are distinct, otherwise either $2x\in Nil(R)$ or $n_i=n_j$ for $i\neq j$. Let $|Nil(R)|<3$, since $|Nil(R)|$ divides $|R|$, so $|Nil(R)|=1$ and hence $gr(G(R))$ is infinite.
\end{proof}
\begin{Thm}
  Let $R$ be a commutative ring of even order and $|Nil(R)|\geq 3$ then $gr(G(R))$ is $3$.
\end{Thm}
\begin{proof}
  By Lemma \ref{Gmain1}, there exists an element $x\in R\setminus Nil(R)$ such that $2x\in Nil(R)$. Since $|Nil(R)|\geq 3$, so there exist elements $n_1,n_2,n_3\in Nil(R)$ such that all are distinct. It is clear that $n_3+x\longrightarrow n_2+x \longrightarrow n_1+x \longrightarrow n_3+x$ be a cycle in $G(R)$. Clearly $n_i+x\in R\setminus Nil(R)$, for $1\leq i \leq 3$ and all are distinct. Hence $gr(G(R))=3$.
\end{proof}
\section{Clique number}
\begin{Thm}\label{C1}
  For the graph $G(\mathbb{Z}_n)$ the following hold:
  \begin{enumerate}
    \item If $n$ is odd, then the clique number of $G(\mathbb{Z}_n)$ is $2$.
    \item If $n$ is even, then the clique number of $G(\mathbb{Z}_n)$ is $|Nil(\mathbb{Z}_n)|$ provided $|Nil(\mathbb{Z}_n)|\geq 2$ and the clique number of $G(\mathbb{Z}_n)$ is $2$ if $|Nil(\mathbb{Z}_n)|=1$.
  \end{enumerate}
\end{Thm}
\begin{proof}
  Proof follows from Corollary \ref{main2} and Corollary \ref{main3}.
\end{proof}
Next we generalise the Theorem \ref{C1} to arbitrary finite commutative ring $R$.
\begin{Thm}
  Let $R$ be a finite commutative ring
  \begin{enumerate}
    \item If $|R|$ is odd, then $\omega (G(R))=2$.
    \item If $|R|$ is even, then $\omega (G(R))=|Nil(R)|$ provided $|Nil(R)|\geq 2$, otherwise $\omega (G(R))=2$.
  \end{enumerate}
\end{Thm}
\begin{proof}
  \begin{enumerate}
    \item Clear from Theorem \ref{GThmo}.
    \item From definition of nilpotent graph, it is clear that $\omega (G(R))\leq |Nil(R)|$. Next we claim that $\omega(G(R))= |Nil(R)|$. By Lemma \ref{Gmain1}, there exists $x\in R\setminus Nil(R)$ such that $2x\in Nil(R)$. Let $|Nil(R)|=n$ and $a_i's$ are distinct elements of $Nil(R)$, for $0\leq i\leq n-1$. Consider $A=\{x+a_i\,\,:\,\, a_i\in Nil(R)\}$. Observe that all pair of distinct elements of $A$ are adjacent and $A\subseteq R\setminus Nil(R)$. Hence vertex set $A$ forms a complete subgraph of $G(R)$. If $|Nil(R)|=1$ then obviously $\omega (G(R))=2$.
  \end{enumerate}
\end{proof}

\section{Spectrum, Laplacian spectrum and signless Laplacian spectrum of $G(\mathbb{Z}_n)$}
For a graph $G$, let $A(G)$ and $D(G)$ are respectively adjacency matrix and degree matrix of $G$. Then Laplacian and signless Laplacian matrix are given by $L(G)=D(G)-A(G)$ and $Q(G)=D(G)+A(G)$ respectively. 
The collection of eigen values of adjacency matrix, Laplacian matrix and signless Laplacian matrix are called spectrum, Laplacian spectrum and signless Laplacian spectrum respectively. We denote spectrum, Laplacian spectrum and signless Laplacian spectrum of $G$ by $Spac(G)$, $L$-$Spac(G)$ and $Q$-$Spac(G)$ respectively. We write $Spac(G)=\{(a_1)^{l_1}, (a_2)^{l_2}, \cdot\cdot\cdot , (a_p)^{l_p}\}$, $L$-$Spac(G)=\{(b_1)^{m_1}, (b_2)^{m_2}, \cdot\cdot\cdot , (b_q)^{m_q}\}$ and $Q$-$Spac(G)=\{(c_1)^{n_1}, (c_2)^{n_2}, \cdot\cdot\cdot , (c_r)^{n_r}\}$, where $a_1, a_2, \cdot\cdot\cdot , a_p$ are eigenvalues of $A(G)$ with multiplicities $l_1, l_2, \cdot\cdot\cdot , l_p$; $b_1, b_2, \cdot\cdot\cdot , b_q$ are eigenvalues of $L(G)$ with multiplicities $m_1, m_2, \cdot\cdot\cdot , m_q$ and $c_1, c_2, \cdot\cdot\cdot, c_r$ are eigenvalues of $Q(G)$ with multiplicities $n_1, n_2, \cdot\cdot\cdot , n_r$ respectively.
 It is well known that the spectrum, laplacian spectrum and signless laplacian spectrum of the complete bipartite graph $K_{n,n}$ are given by $\{(n)^1, (-n)^1, (0)^{2n-2}\}$, $\{(2n)^1, (0)^1, (n)^{2n-2}\}$ and $\{(2n)^1, (0)^1, (n)^{2n-2}\}$. Also the spectrum, laplacian spectrum and signless laplacian spectrum of the complete graph $K_n$ are given by $\{(-1)^{n-1}, (n-1)^1\}$, $\{(0)^1, (n)^{n-1}\}$ and $\{(2n-2)^1, (n-2)^{n-1}\}$.
 \begin{Thm}
   For $G(\mathbb{Z}_n)$, if $|Nil(\mathbb{Z}_n)|=t$ then the following hold:
   \begin{enumerate}
      \item If $n=2^{r_0}p_1^{r_1}p_2^{r_2}\cdot\cdot\cdot p_k^{r_k}$, where $p_i$'s are distinct odd primes. Then
     \begin{itemize}
        \item $Spec(G(\mathbb{Z}_n))=\{(t)^{p_1p_2\cdot\cdot\cdot p_k-1}, (-t)^{p_1p_2\cdot\cdot\cdot p_k-1}, (0)^{(2t-2)(p_1p_2\cdot\cdot\cdot p_k-1)}, (-1)^{t-1}, (t-1)^1 \}$.
        \item $L$-$Spec(G(\mathbb{Z}_n))=\{(2t)^{p_1p_2\cdot\cdot\cdot p_k-1}, (0)^{p_1p_2\cdot\cdot\cdot p_k-1}, (t)^{(2t-2)(p_1p_2\cdot\cdot\cdot p_k-1)}, (0)^1, (t)^{t-1}\}$
        \item $Q$-$Spec(G(\mathbb{Z}_n))=\{(2t)^{p_1p_2\cdot\cdot\cdot p_k-1}, (0)^{p_1p_2\cdot\cdot\cdot p_k-1}, (t)^{(2t-2)(p_1p_2\cdot\cdot\cdot p_k-1)}, (2t-2)^1, (t-2)^{t-1}\}$.
     \end{itemize}

     \item If $n=p_1^{r_1}p_2^{r_2}\cdot\cdot\cdot p_k^{r_k}$, where $p_i$'s are distinct odd primes then
      \begin{itemize}
        \item $Spec(G(\mathbb{Z}_n))=\{(t)^{[\frac{p_1p_2\cdot\cdot\cdot p_k}{2}]}, (-t)^{[\frac{p_1p_2\cdot\cdot\cdot p_k}{2}]},(0)^{(2t-2)({[\frac{p_1p_2\cdot\cdot\cdot p_k}{2}]})}\}$.
        \item $L$-$Spec(G(\mathbb{Z}_n))=\{\{(2t)^{[\frac{p_1p_2\cdot\cdot\cdot p_k}{2}]}, (0)^{[\frac{p_1p_2\cdot\cdot\cdot p_k}{2}]}, (t)^{(2t-2)({[\frac{p_1p_2\cdot\cdot\cdot p_k}{2}]})}\}\}$
        \item $Q$-$Spec(G(\mathbb{Z}_n))=\{\{(2t)^{[\frac{p_1p_2\cdot\cdot\cdot p_k}{2}]}, (0)^{[\frac{p_1p_2\cdot\cdot\cdot p_k}{2}]}, (t)^{(2t-2)({[\frac{p_1p_2\cdot\cdot\cdot p_k}{2}]})}\}\}$.
      \end{itemize}

   \end{enumerate}
 \end{Thm}
 \begin{proof}
   \begin{enumerate}
     \item Let $s=2p_1p_2\cdot\cdot\cdot p_k$ and consider the partition of $V(G(\mathbb{Z}_n))$,
     $P_1=\{1,2, \cdot \cdot \cdot , s-1\}$, $P_2=\{s+1, s+2, \cdot \cdot \cdot, 2s-1\}$, $\cdot \cdot \cdot$ $P_l=\{n-s+1, n-s+2, \cdot \cdot \cdot, n-1\}$. Observe that $|Nil(\mathbb{Z}_n)|=l$. Let us collect $\frac{s}{2}, 3\frac{s}{2}, \cdot \cdot \cdot (2l-1)\frac{s}{2}$ from the partition $P_1, P_2, \cdot \cdot \cdot, P_l$ respectively and say $L=\{\frac{s}{2}, 3\frac{s}{2}, \cdot \cdot \cdot (2l-1)\frac{s}{2}\}$. Consider $Q_i=P_i\setminus \{(2i-1)\frac{s}{2}\}$, for $1\leq i \leq l$ and so $|Q_i|$ is even. Since every pair of elements of $L$ are adjacent and $2a\in Nil(\mathbb{Z}_n)$, for all $a\in L$, hence $span(L)$ in $G(\mathbb{Z}_n)$ is a complete subgraph of $G(\mathbb{Z}_n)$. Next for $1\leq i \leq s-1$, define $B_{-i}=\{m-i\,\,;\,\, m\in Nil(\mathbb{Z}_n)\}$ and $B_{+i}=\{m+i\,\,;\,\, m\in Nil(\mathbb{Z}_n)\}$. Then proceeding similar to the proof of Theorem \ref{main}$(ii)$, we conclude that if $x$ and $y$ are adjacent with $x\in B_{+j}$ for some $1\leq j\leq s-1$, then $y\in B_{-j}$. Hence we get $G(\mathbb{Z}_n)$ consists of $s-1$ complete bipartite subgraphs $K_{|Nil(\mathbb{Z}_n)|, |Nil(\mathbb{Z}_n)|}$ and one complete subgraph $K_{Nil(\mathbb{Z}_n)}$ and all subgraphs mentioned are disjoint. Hence the result follows.
     \item It follows from Corollary \ref{main2}.
   \end{enumerate}
 \end{proof}

Next we study the spectrum, Laplacian spectrum and signless Laplacian spectrum of a ring of odd order.
\begin{Thm}\label{Specodd1}
  If $R$ be a ring of odd order then the graph $G(R)$ is bipartite having disjoint complete bipartite subgraphs.
\end{Thm}
\begin{proof}
  Clearly $|Nil(R)|$ is odd as $R$ is a commutative ring. Let $Nil(R)=\{n_1, n_2, \cdot \cdot \cdot, n_k\}$. For $x\in R\setminus Nil(R)$, consider $A_{x^+}=\{n_i+x\,\,:\,\, 1\leq i \leq k\}$ and $A_{x^-}=\{n_i-x\,\,:\,\, 1\leq i \leq k\}$. Observe that every element of $A_{x^+}$ is adjacent to every element of $A_{x^-}$ and no two elements of $A_{x^+}$ or no two elements of $A_{x^-}$ are adjacent by Lemma \ref{Gcor1}. So $A_{x^+}\cup A_{x^-}$ induces a complete bipartite subgraph of $G(R)$. Let $y\in R\setminus Nil(R)$ with $x\neq y$. \\
  Case I: Let $x+y\in Nil(R)$, then $y=n_i-x$, for some $1\leq i \leq k$. So we get $A_{x^+}=A_{y^-}$ and $A_{x^-}=A_{y^+}$. \\
  Case II: Let $x+y\in R\setminus Nil(R)$. If $x-y\in Nil(R)$ then $x=n_j+y$, for some $1\leq i \leq k$ and we get $A_{x^+}=A_{y^+}$ and $A_{x^-}=A_{y^-}$. If $x-y\in R\setminus Nil(R)$ then $(A_{x^+}\cup A_{x^-})\cap (A_{y^+}\cup A_{y^-})=\emptyset$.\\
  So for distinct $x,y\in R\setminus Nil(R)$, subgraph induced by $A_{x^+}\cup A_{x^-}$ and subgraph induced by $A_{y^+}\cup A_{y^-}$ are either identical or disjoint. Hence the graph $G(R)$ is bipartite having disjoint complete bipartite subgraphs.
\end{proof}

\begin{Thm}\label{Specodd2}
  If $R$ be a ring of odd order, then the number of disjoint complete bipartite subgraphs of $G(R)$ is $\frac{|R|-|Nil(R)|}{2|Nil(R)|}$.
\end{Thm}
\begin{proof}
 Let $x\in R\setminus Nil(R)$ and $Nil(R)=\{n_1, n_2, \cdot \cdot \cdot, n_k\}$. From the above proof it is enough to show $|A_{x^+}\cup A_{x^-}|=2|Nil(R)|$. Clearly $n_i+x=n_j+x$ or $n_i-x=n_j-x$ implies $i=j$. Also $n_i-x=n_j+x$ implies $2x\in Nil(R)$, a contradiction by Lemma \ref{Gcor1}. Hence the result follows.
\end{proof}
\begin{Thm}
  If $R$ be a ring of odd order then 
  \begin{itemize}
    \item $Spec(G(R))=\{(|Nil(R)|)^m, (-|Nil(R)|)^m, (0)^{2m(|Nil(R)|-1)}\}$
    \item $L$-$Spec(G(R))=\{(2|Nil(R)|)^m, (0)^m, (|Nil(R)|)^{2m(|Nil(R)|-1)}\}$
    \item $Q$-$Spec(G(R))=\{(2|Nil(R)|)^m, (0)^m, (|Nil(R)|)^{2m(|Nil(R)|-1)}\}$.
  \end{itemize}
  Where $m=\frac{|R|-|Nil(R)|}{2|Nil(R)|}$.
\end{Thm}
\begin{proof}
  Proof follows from Theorem \ref{Specodd1} and Theorem \ref{Specodd2}, as $Spec(K_{n,n})=\{(n)^1, (-n)^1, (0)^{2n-2}\}$, $L$-$Spec(K_{n,n})=\{(2n)^1, (0)^1, (n)^{2n-2}\}$ and $Q$-$Spec(K_{n,n})=\{(2n)^1, (0)^1, (n)^{2n-2}\}$.
\end{proof}
\section{Dominating number}
For a graph $G$, a subset $D$ of the vertex set of $G$ is said to be a dominating set of $G$ if every vertex not in $D$ is adjacent to at least one member of $D$. The dominating number is the number of vertices in the smallest dominating set for $G$. In this section we study the dominating number of $G(\mathbb{Z}_n)$ and $G(R)$, where $R$ is a ring of odd order.

\begin{Thm}
  For $G(\mathbb{Z}_n)$, the following hold:
  \begin{enumerate}
    \item If $n$ is odd, then the dominating number is $\frac{|R|-|Nil(R)|}{2}$.
    \item If $n$ is even, then the dominating number is $\frac{|R|-2|Nil(R)|}{2}+1$.
  \end{enumerate}
\end{Thm}
\begin{proof}
  \begin{enumerate}
    \item From Corollary \ref{main2}, if $n$ is odd, then $G(\mathbb{Z}_n)$ is a disjoint union of $\frac{|R|-|Nil(R)|}{2|Nil(R)|}$ complete bipartite subgraphs $K_{|Nil(\mathbb{Z}_n)|, |Nil(\mathbb{Z}_n)|}$. Hence the result follows.
    \item Similarly it follows from Corollary \ref{main3}.
  \end{enumerate}
\end{proof}
\begin{Thm}
  If $R$ is a finite commutative ring of odd order then the dominating number is $\frac{|R|-|Nil(R)|}{2}$.
\end{Thm}
\begin{proof}
  Proof follows from Theorem \ref{Specodd1} and Theorem \ref{Specodd2}.
\end{proof}
\section{Chromatic index}
Vizings Theorem \cite{gt} says that $\bigtriangleup \leq \chi^\prime (G)\leq \bigtriangleup +1$, where $\bigtriangleup$ be the maximum vertex degree of $G$. Graph satisfying $\chi^\prime (G)=\bigtriangleup$ are called graphs of class $1$ and those with $\chi^\prime (G)=\bigtriangleup +1$ are called graphs of class $2$.
\begin{Thm}
  Nilpotent graph of any ring $R$ is of class $1$.
\end{Thm}
\begin{proof}
  Put colour $x+y$ for an edge $xy$ of $G(R)$. Let $C=\{x+y\,:\,xy$ is an edge of $G(R)\}$ then $C$ is the set of colours of $G(R)$. Therefore $G(R)$ has a $|C|$-$edge$ colouring, so $\chi^\prime (G(R))\leq |C|$. But $C\subseteq Nil(R)$ and $\chi^\prime (G(R))\leq |C|\leq |Nil(R)|$. Also $\bigtriangleup \leq |Nil(R)|$, hence by Vizings Theorem we get $\chi^\prime (G(R))=|Nil(R)|=\bigtriangleup$.
\end{proof}

\end{document}